\documentclass[11pt]{article}

\usepackage{amsmath,amssymb,amsthm,amscd,epsfig}
\theoremstyle{plain}

\newtheorem{thm}{Theorem}

\newtheorem*{unthm}{Theorem}

\newtheorem{cor}{Corollary}

\newtheorem{pro}{Proposition}

\theoremstyle{definition}

\theoremstyle{remark}

\begin{document}

\title{Ergodic lifts and overlap numbers}

\author{Eugen Mihailescu}
\date{}
\maketitle

\begin{abstract}

We study skew product  lifts  and overlap numbers for equilibrium measures $\mu_\psi$ of H\"older continuous potentials $\psi$ on such lifts. We find computable formulas and estimates for the overlap numbers in several concrete significant cases of systems with overlaps. In particular we obtain iterated systems which are asymptotically irrational-to-1 and absolutely continuous on their limit sets. Then we look into the general structure of the Rokhlin conditional measures of $\mu_\psi$ with respect to different fiber partitions associated to the lift $\Phi$, and find relations between them. Moreover we prove an estimate on the box dimension of a certain associated invariant measure $\nu_\psi$ on the limit set $\Lambda$ by using the overlap number of $\mu_\psi$.  

\end{abstract}



\section{Introduction.}

In this paper we study and give several formulas and applications of overlap numbers of equilibrium measures over iterated systems. These overlap numbers were introduced in \cite{MU-JSP2016}, and represent asymptotic averages of the numbers of generic preimages in the limit set.

Consider thus a finite iterated function system (IFS) $\mathcal S = \{\phi_i, i \in I\}$, where the contractions $\phi_i$ are conformal and injective on an open set $U \subset \mathbb R^d$. Denote by $\Sigma_I^+$ the one-sided symbolic space $\{\omega = (\omega_1, \omega_2, \ldots), \omega_i \in I, i \ge 1\}$, with the canonical metric and topology, and with the shift map $\sigma: \Sigma_I^+ \to \Sigma_I^+$. Denote also by $[\omega_1 \ldots \omega_n]$ the cylinder $\{\eta \in \Sigma_I^+, \eta_1 = \omega_1, \ldots, \eta_n = \omega_n\}$. In general we denote by $\phi_{i_1\ldots i_p} := \phi_{i_1} \circ \phi_{i_2} \circ \ldots \circ \phi_{i_p}$ for any $p \ge 1, i_1, \ldots, i_p \in I$, and where $\phi_{i_1 i_2 \ldots}$ is the point given as the intersection of the descending sequence of sets $\phi_{i_1\ldots i_p}(U)$, when $p \to \infty$. 
We  denote by $\Lambda$ the \textit{limit set} of $\mathcal S$, and consequently
$$\Lambda = \pi(\Sigma_I^+), $$ where $\pi: \Sigma_I^+ \to \Lambda, \ \pi(\omega) = \phi_{\omega_1 \omega_2 \ldots}, \ \omega \in \Sigma_I^+$, is the canonical projection to the limit set.  
We then consider the following skew product map, which we call \textit{the lift of} $\mathcal S$, namely:
$$\Phi: \Sigma_I^+ \times \Lambda \to \Sigma_I^+ \times \Lambda, \ \Phi(\omega, x) = (\sigma \omega, \phi_{\omega_1}(x)) \ \text{for} \ (\omega, x) \in \Sigma_I^+ \times \Lambda$$
In general, for any $n \ge 1$,  the $n$-th iterate of $\Phi$ looks like:
$$\Phi^n(\omega, x) = (\sigma^n(\omega), \phi_{\omega_n\ldots \omega_1}(x)), \  (\omega, x) \in \Sigma_I^+ \times \Lambda
$$
It is clear that,  due to the expansion of the shift map $\sigma$ and the contraction of $\phi_i, i \in I$, the skew product $\Phi$ has a  hyperbolic character. Consider now a H\"older continuous potential $\psi: \Sigma_I^+ \times \Lambda \to \mathbb R$. Then there exists a unique \textit{equilibrium measure} $\mu_\psi$ on $\Sigma_I^+ \times \Lambda$, i.e a measure which maximizes in the Variational Principle for Pressure (for eg \cite{KH}, \cite{KS}, \cite{Pe}, \cite{Wa}). 
More precisely, if 
$P_\Phi: \mathcal C(\Sigma_I^+\times \Lambda, \mathbb R) \longrightarrow \mathbb R$ is the pressure functional for $\Phi$ on $\Sigma_I^+ \times \Lambda$, then $\mu_\psi$ is the unique $\Phi$-invariant probability for which $$P_\Phi(\psi) = h_\Phi(\mu_\psi) + \int_{\Sigma_I^+\times \Lambda}\psi \ d\mu_\psi = \sup \{h_\Phi(\mu) + \int_{\Sigma_I^+ \times \Lambda}\psi \ d\mu, \  \mu \ \Phi-\text{invariant probability on} \ \Sigma_I^+ \times \Lambda\},$$ 
where $h_\Phi(\mu)$ is the measure-theoretic entropy of $\mu$ with respect to $\Phi$. It is known that $\mu_\psi$ is an ergodic measure. Denote also by $$\nu_\psi:= \pi_{2*} \mu_\psi,$$ the projection of $\mu_\psi$ on the second coordinate. Then $\nu_\psi$ is a probability measure on the limit set $\Lambda$, and we want to study the metric properties of this measure. Notice that, in general, $\nu_\psi$ is not equal to the classical projection $ \pi_*(\pi_{1*}(\mu_\psi))$ of the measure $\mu_\psi$ from  $\Sigma_I^+\times \Lambda$ to the limit set $\Lambda$. 

For a $\Phi$-invariant probability measure $\mu$ on $\Sigma_I^+ \times \Lambda$, we define as usual its \textit{Lyapunov exponent},
$$\chi(\mu) = \int_{\Sigma_I^+ \times \Lambda}  - \log|\phi'_{\omega_1}(x)| \ d\mu(\omega, x)$$
Also let us notice that since the skew product $\Phi$ is contracting in the second coordinate, we have that the entropy of $\mu$ is actually equal to the entropy of its projection on the first coordinate, 
$$
h_\Phi(\mu) = h_\sigma(\pi_{1*}\mu)
$$
 In \cite{Ru-folding}, \cite{Ru-survey99} was introduced a notion of folding entropy of a measure, denoted in our case by $F_\Phi(\mu)$ which  is defined as the conditional entropy $H_{\mu}(\epsilon|\Phi^{-1}\epsilon)$.
If $\Phi^{-1}(\epsilon)$ is the measurable partition of $\Sigma_I^+ \times \Lambda$ with the fibers of $\Phi$, and if $\mu$ is an $\Phi$-invariant probability measure on $\Sigma_I^+\times \Lambda$, then we obtain a system of conditional measures of $\mu$ denoted by $(\mu_{(\omega, x)})_{(\omega, x) \in \Sigma_I^+\times \Lambda}$, where $\mu_{(\omega, x)}$ is a probability supported on the finite fiber $\Phi^{-1}(\omega, x)$. 

Also let us recall that the \textit{Jacobian} of an invariant measure introduced in \cite{Pa}, as the local Radon-Nikodym derivative of the push-forward with respect to the measure. If  $\mu$ is a $\Phi$-invariant measure on $\Sigma_I^+ \times \Lambda$, then we denote by $J_\Phi(\mu)$ its Jacobian; from definition, $J_\Phi(\mu) \ge 1$. 
In our case it follows that the folding entropy and Jacobian are related by

\begin{equation}\label{foldingentropy}
F_\Phi(\mu_\psi) = \int_{\Sigma_I^+ \times \Lambda} \log J_\Phi(\mu)(\omega, x) \ d\mu(\omega, x)
\end{equation}

\

In \cite{MU-JSP2016} we introduced a notion of \textit{overlap number} $o(\mathcal S, \mu_\psi)$ for an equilibrium measure $\mu_\psi$ of a H\"older continuous potential on the lift $\Sigma_I^+ \times \Lambda$. This overlap number is an average asymptotic number of the $\mu_\psi$-generic preimages in $\Lambda$ (since the points in $\Lambda$ can be covered multiple times by the images $\phi_{i_1\ldots i_m}(\Lambda)$ if the system $\mathcal S$ has overlaps). More precisely, for an arbitrary number $\tau>0$ denote the set of $\mu_\psi$-generic preimages having the same iterates as $(\omega, x)$ by $$\Delta_n((\omega, x), \tau, \mu_\psi) = \{(\eta_1, \ldots, \eta_n) \in I^n, \exists y \in \Lambda, \phi_{\omega_n\ldots \omega_1}(x) = \phi_{\eta_n\ldots \eta_1}(y),  \ |\frac{S_n\psi(\eta, y)}{n} - \int_{\Sigma_I^+ \times \Lambda} \psi\ d\mu_\psi| < \tau\},$$
where $(\omega, x) \in \Sigma_I^+ \times \Lambda$ and $S_n\psi(\eta, y)$ is the consecutive sum of $\psi$ with respect to the skew product $\Phi$. Denote  by $$b_n((\omega, x), \tau, \mu_\psi):= Card\Delta_n((\omega, x), \tau, \mu_\psi)$$
Then, in \cite{MU-JSP2016} we proved that the following limit exists and defines the \textit{overlap number} of $\mu_\psi$,
$$o(\mathcal S, \mu_\psi) = \mathop{\lim}\limits_{\tau \to 0} \mathop{\lim}\limits_{n \to \infty} \frac 1n \int_{\Sigma_I^+ \times \Lambda} \log b_n((\omega, x), \tau, \mu_\psi) \ d\mu_\psi(\omega, x)$$ 
Moreover in the same paper \cite{MU-JSP2016} we proved a connection between the overlap number and the folding entropy of $\mu_\psi$, namely,
\begin{equation}\label{oS}
o(\mathcal S, \mu_\psi) = \exp(F_\Phi(\mu_\psi))
\end{equation} 
We found the following estimate for the Hausdorff dimension of the projection $\nu_\psi$  of $\mu_\psi$ on the second coordinate. Recall that $\pi_2: \Sigma_I^+ \times \Lambda \to \Lambda, \pi_2(\omega, x) = x$. The measure $\nu_\psi$ is not usually equal to the other projection $ \pi_*(\pi_{1*}(\mu_\psi))$ of $\mu_\psi$ from $\Sigma_I^+\times \Lambda$ onto the limit set $\Lambda$. 

\begin{unthm}[\cite{MU-JSP2016}]
If $\mathcal S$ is a finite conformal iterated function system as above, and if $\psi: \Sigma_I^+ \times \Lambda \to \mathbb R$ is H\"older continuous with its equilibrium measure $\mu_\psi$, and if $\nu_\psi := \pi_{2*}(\mu_\psi)$, then
$$HD(\nu_\psi) \le t(\mathcal S, \psi),$$
where $t(\mathcal S, \psi)$ is the unique zero of the pressure function $t \to P_\sigma(t\log|\phi_{\omega_1}(\pi(\sigma\omega))| - \log o(\mathcal S, \mu_\psi))$. 
\end{unthm}

\

In the current paper,  in Theorems \ref{garsia} and \ref{pisot}, we compute/estimate overlap numbers in several concrete significant cases, namely for Bernoulli convolution systems associated 
to reciprocals of Garsia and Pisot numbers (see \cite{Ga}, \cite{PU}). In particular we obtain examples of systems which asymptotically are \textit{irrational}-to-1 on their limit sets. More precisely, for any $n \ge 1$, we obtain systems with overlaps which asymptotically are $\sqrt[n]{2^{n-1}}$ -to-1 and absolutely continuous on their limit sets. 

Then in Proposition \ref{o} and Corollaries \ref{levelp} and \ref{partial}, we compute the overlap numbers for systems with eventual exact overlaps, and estimate the overlap numbers for systems with partial overlaps. 

Next, for general systems and equilibrium measures, we compare in Theorem \ref{Hfold} the conditional measures obtained from $\mu_\psi$ by taking certain special measurable partitions of the skew product into fibers. We apply this to find a formula for overlap numbers, by using families of conditional measures which may be easier to find in certain cases (for example for Bernoulli measures). 

Then in Theorem \ref{mthm}
we find an upper bound for the lower box dimension of $\nu_\psi$,  with the help of the overlap number of $\mu_\psi$, and using the Bounded Distortion Property for conformal systems of contractions. 
We give a \textit{constructive} method to find a set of large $\nu_\psi$-measure in $\Lambda$ whose lower box dimension is bounded with the help of overlap numbers, namely is less than $$\frac{h_\sigma(\pi_{1*}(\mu_\psi)) - \log o(\mathcal S, \mu_\psi))}{|\chi(\mu_\psi)|}$$ This is done by careful estimates of the proportion of the measure of generic points within the measure of balls, by Jacobians of iterates, and employing the distribution of regular points from the Borel Density Lemma. In general for estimates of box dimensions one needs covers with balls of same radii (see \cite{Ba}, \cite{Pe}), unlike for Hausdorff dimension; thus generic points are important. We then give in Corollaries \ref{o1} and \ref{o2}  applications to estimates for box dimensions for projection measures, which work in particular for Bernoulli convolutions.

\section{Formulas and estimates for overlap numbers.}\label{formulas}

First we study the topological overlap number for various systems with overlaps. The \textit{topological overlap number} of a conformal iterated system $\mathcal S = \{\phi_i, i \in I\}$ is defined (see \cite{MU-JSP2016}) as the overlap number of the measure of maximal entropy $\mu_{max}$ for $\Phi$ on $\Sigma_I^+ \times \Lambda$, and denoted by $o(\mathcal S)$. Thus $$o(\mathcal S) = o(\mathcal S, \mu_{max})$$

Consider now a probabilistic vector $\bf p = (p_1, \ldots, p_{|I|})$ and its associated Bernoulli measure $\mu_{\bf p}^+$ on $\Sigma_I^+$. Then the classical projection of $\mu_{\bf p}^+$ on the limit set $\Lambda$ of $\mathcal S$ is $\pi_*\mu_{\bf p}^+$. The Bernoulli measure $\mu_{\bf p}^+$ is the equilibrium measure with respect to $\sigma$ of the potential $g:\Sigma_I^+ \to \mathbb R, \ g(\omega) = \log p_{\omega_1}, \ \omega \in \Sigma_I^+$. 
Let $\psi:= g \circ \pi_1: \Sigma_I^+ \times \Lambda \to \mathbb R$, and $\mu_\psi$ be its equilibrium measures with respect to $\Phi$. Then we proved in \cite{MU-JSP2016} that for this choice of $\psi$,  \ $\pi_{2*} \mu_\psi = \pi_*\pi_{1*}\mu_\psi$. On the other hand, notice that from estimates of equilibrium measures on Bowen balls, it follows that for some constant $r_0$, $\mu_\psi([\omega_1\ldots \omega_n]\times B(x, r_0)) \approx e^{S_n \psi(\omega, x) - nP_\Phi(\psi)},$
where the comparability constant is independent of $n, x, \omega$. Thus by summing up, $$\mu_\psi([\omega_1\ldots \omega_n]\times \Lambda) \approx e^{S_ng(\omega)-nP_\sigma(g)},$$ 
since $\Phi$ is contracting in the second coordinate and since $\psi$ depends only on $\omega$. Denote $\mu_{g\circ\pi_1}$ by $\mu_{\bf p}$, which can be considered a lift of $\mu_{\bf p}^+$ to $\Sigma_I^+ \times \Lambda$. So $\pi_{1*}\mu_{\bf p}$ satisfies the same estimates on cylinders as the Bernoulli measure $\mu_{\bf p}^+$, and thus from above, we obtain $\pi_{1*}\mu_{\bf p} = \mu_{\bf p}^+$. Therefore,
\begin{equation}\label{bernoulli}
\pi_{2*}\mu_{\bf p} = \pi_*\mu_{\bf p}^+
\end{equation}
In particular, if $\mu_{max}^+$ denotes the measure of maximal entropy for the shift on $\Sigma_I^+$, i.e the Bernoulli measure associated to the probability vector $(\frac{1}{|I|}, \ldots, \frac{1}{|I|})$), we obtain
\begin{equation}\label{max}
\pi_{2*}\mu_{max} = \pi_*\mu_{max}^+
\end{equation}
We showed in \cite{MU-JSP2016} that,  if $\pi: \Sigma_I^+ \to \Lambda$ is the canonical projection to the limit set of $\mathcal S$ and if $$\beta_n(x):= Card\{(\eta_1, \ldots, \eta_n) \in I^n, x \in \phi_{\eta_1\ldots \eta_n}(\Lambda)\}, \ n \ge 1,$$   then the topological overlap number of $\mathcal S$ is given by the formula:
\begin{equation}\label{otop}
o(\mathcal S) = \exp\big(\mathop{\lim}\limits_{n \to \infty} \frac 1n \int_{\Sigma_I^+} \log \beta_n(\pi\omega) \ d\mu_{max}^+(\omega)\big)
\end{equation}
\

\textbf{3.1.} Consider the IFS $\mathcal S_\lambda = \{\phi_{-1}, \phi_1\}$, where $\phi_{-1}(x) = \lambda x -1, \ \phi_1(x) = \lambda x+1$. When $\lambda \in (\frac 12, 1)$ this system has overlaps, and its limit set is the interval $I_\lambda = [-\frac{1}{1-\lambda}, \frac{1}{1-\lambda}]$. When there is no confusion about $\lambda$, this limit set will also be denoted by $\Lambda$.  We consider then the measure of maximal entropy $\mu_{max}$ for $\Phi$ on $\Sigma_2^+\times \Lambda$.

\

 \ \ \ \textbf{3.1a.} Let us look first at reciprocals of \textit{Garsia numbers}. A number $\gamma$ is called a \textit{Garsia number} if it is an algebraic integer in $(1, 2)$ whose minimal polynomial has constant coefficient $\pm 2$ and so that $\gamma$ and all of its conjugates have absolute value strictly greater than 1 (see \cite{Ga}). Examples of such minimal polynomials are $x^{n+p} - x^n -2$ for $n, p \ge 1$, with $\max\{p, n\} \ge 2$. For instance $2^{\frac 1n}, n \ge 2,$ are  Garsia numbers. We prove the following:

 \begin{thm}\label{garsia}
 The topological overlap number $o(\mathcal S_\lambda)$ of the system $\mathcal S_\lambda$ for $\lambda \in (\frac 12, 1)$ with $\frac 1\lambda$ a Garsia number, is equal to $ 2\lambda$.
 \end{thm}
 
 \begin{proof}
 Recall that the limit set of $\mathcal S_\lambda$ is the interval $I_\lambda = [-\frac{1}{1-\lambda}, \frac{1}{1-\lambda}]$. 
 From \cite{Ga} it follows that, if $\lambda$ is the reciprocal of a Garsia number, then all $2^n$ sums of type $\mathop{\sum}\limits_{0}^{n-1} \pm \lambda^k$ are distinct and at least $\frac{C}{2^n}$ apart, for some constant $C>0$. 
Let us  order increasingly these $2^n$ numbers  $\mathop{\sum}\limits_{0}^{n-1} \pm \lambda^k$, and denote them by $\zeta_1, \ldots, \zeta_{2^n}$. Hence these points $\zeta_i$ are distinct, and 
\begin{equation}\label{di}
|\zeta_i - \zeta_j| \ge \frac{C}{2^n}, i \ne j
\end{equation}
Now the numbers of type $\zeta_j + \mathop{\sum}\limits_{k \ge n} r_k\lambda^k$, where $\zeta_j = \mathop{\sum}\limits_{0 \le k \le n-1} \omega_k\lambda^k$ and $\omega_k \in \{-1, 1\}$, form the interval $I_j:= \pi([\omega_0, \ldots \omega_{n-1}]$. 
The length of  $I_j$ is $C_1 \lambda^n$, for some fixed constant $C_1>0$. Since $\lambda > \frac 12$, it follows from (\ref{di}) that any interval $I_j$ contains at least $C_2(2\lambda)^n$ points $\zeta_j$ and at most $C_3 (2\lambda)^n$ points $\zeta_j$, for some constants $C_3 > C_2 >0$. With the possible exception of an interval $J_1$ of length $C_4\lambda^n$ with left endpoint $-\frac{1}{1-\lambda}$ (i.e the left endpoint of $I_\lambda$), and an interval $J_2$ of same length with right endpoint $\frac{1}{1-\lambda}$ (i.e the right endpoint of $I_\lambda$), we see that  any point $x$ belongs to  at least $C_5(2\lambda)^n$ intervals $I_j$ and to at most $C_6(2\lambda)^n$ intervals $I_j$, where the constants $C_1, \ldots, C_6$ do not depend on $n$. \newline 
Recall that $I_j = \pi([\omega_0, \ldots, \omega_{n-1}]$ for some $\omega_k \in \{-1, 1\}, 0 \le k \le n-1$, and that $\mu_{max}^+([\omega_0, \ldots, \omega_{n-1}]) = \frac{1}{2^n}$, where $\mu_{max}^+$ is the measure of maximal entropy on $\Sigma_2^+$. 
From above (\ref{otop}) we know that,
$$o(\mathcal S_\lambda) = \exp (\mathop{\lim}\limits_n \frac 1n\int_{\Sigma_2^+} \log \beta_n(\pi\omega) \ d\mu_{max}^+(\omega)), $$
where $\beta_n(x):= Card\{(\eta_0, \ldots, \eta_{n-1}) \in \{-1, 1\}^n, x \in \phi_{\eta_0\ldots\eta_{n-1}}(\Lambda_\lambda)\}$ for $x \in \Lambda_\lambda$ and $n \ge 1$.
But from above, we see that for $x$ outside the intervals $J_1, J_2$ of length $C_4\lambda^n$ at the endpoints of $I_\lambda$, $$C_5 (2\lambda)^n \le \beta_n(x) \le C_6(2\lambda)^n$$
Thus from the last estimate on $\beta_n(x)$ on the complement of $J_1 \cup J_2$, and using that $\mu_{max}([\omega_0, \ldots, \omega_{n-1}]) = \frac{1}{2^n}$, we obtain that for some constant $C_7>0$ (independent of $n$),  $$(2^n-C_7(2\lambda)^n) \cdot n\log(2\lambda) \frac{1}{2^n} \le \int_{\Sigma_2^+} \log \beta_n(\pi \omega) \ d\mu_{max}^+(\omega) \le 2^n\cdot n \log (2\lambda) \cdot \frac{1}{2^n} = n \log (2\lambda) $$
Therefore  $o(\mathcal S_\lambda) = 2\lambda$, since from the last displayed inequalities it follows that,
$$\mathop{\lim}\limits_{n \to \infty} \frac{1}{n} \int_{\Sigma_2^+} \log \beta_n(\pi\omega) \ d\mu_{max}^+(\omega) = \log (2\lambda)$$

\end{proof}

Since for any $n \ge 1$, $2^{\frac 1n}$ is a Garsia number (see \cite{Ga}), we then obtain from Theorem \ref{garsia} a system which asymptotically is $\sqrt[n] {2^{n-1}}$ -to-1; for these examples the projection $\pi_*\mu_{max}^+$ is absolutely continuous, and $\pi_*\mu_{max}^+ = \pi_{2*}\mu_{max}$ from (\ref{max}), hence:

\begin{cor}
For the system $\mathcal S_\lambda$ with $\lambda = 2^{-\frac 1n}$, the topological overlap number is $o(\mathcal S_\lambda) = \sqrt[n] {2^{n-1}}$, and the measure $\pi_{2*}\mu_{\max}$ is absolutely continuous.
\end{cor}

\

\ \ \ \textbf{3.1b.} The second example is  of Bernoulli convolutions with $\lambda$ being the reciprocal of a \textit{Pisot number}.
A Pisot number is by definition an algebraic integer all of whose conjugates are strictly less than 1 in absolute value (for eg \cite{Ga}, \cite{PU}, etc). 
We prove the following.

\begin{thm}\label{pisot}
The topological overlap number of $\mathcal S_\lambda$ for $\lambda\in (\frac 12, 1)$ with $\frac 1\lambda$ a Pisot number, satisfies $$o(\mathcal S_\lambda) \ge 2\lambda >1$$
\end{thm}

\begin{proof}
If $\frac 1\lambda$ is a Pisot number, the distance between any two different polynomial sums of type $P(\omega, \lambda, n)= \mathop{\sum}\limits_{i = 0}^{n-1}\omega_i\lambda^i$ for $\omega \in \Sigma_2^+ = \{-1, 1\}^\infty$, is at least $C\lambda^n$, for some constant $C>0$, which follows from the algebraic properties of $\frac 1\lambda$ (see \cite{Ga}, \cite{PU}). Then  the number $q(n)$ of all possible values of such polynomials $P(\omega, \lambda, n)$, when $n, \lambda$ are fixed, satisfies 
\begin{equation}\label{qn}
q(n) \le C_1 \lambda^{-n},
\end{equation}
for some constant $C_1$ independent of $n$.
Since there are $2^n$ tuples $(\omega_0, \ldots, \omega_{n-1}) \in \{-1, 1\}^n$, but only at most $C_1\lambda^{-n}$ values for polynomials $P(\omega, \lambda, n)$, and since $\lambda > \frac 12$, there must be many equalities between such values. Denote by $V_n(\lambda)$ the set of values of  polynomials $P(\omega, \lambda, n)$, 
\begin{equation}\label{vn}
V_n(\lambda) = \{\alpha_1, \ldots, \alpha_{q(n)}\}, \ \text{where} \  \alpha_1 < \ldots < \alpha_{q(n)},
\end{equation}
where $q(n)$ satisfies (\ref{qn}). We know  that $$\pi([\omega_0, \ldots, \omega_{n-1}]) = \{P(\omega, \lambda, n) + \mathop{\sum}\limits_{i = n}^\infty \omega_i \lambda^i, \ \omega_i \in \{-1, 1\}, i \ge n\},$$
so $\pi([\omega_0, \ldots, \omega_{n-1}])$ is an interval in $\Lambda_\lambda$ of length between $\lambda^n$ and $2\lambda^n$ (depending on its location). Denote by $$N_i:= Card\{(\omega_0, \ldots, \omega_{n-1}) \in \{-1, 1\}^n, P(\omega, \lambda, n) = \alpha_i\}, \ 1 \le i \le q(n)$$
From (\ref{qn}) recall that $|\alpha_i - \alpha_j| \ge C_1\lambda^n$ if $i \ne j$. Since each value $\alpha_i$ is taken $N_i$ times by polynomials $P(\omega, \lambda, n)$,  $1 \le i, j \le q(n)$, it follows that there exists a constant $C_2>0$ so that for all $n \ge 1$,
\begin{equation}\label{betan}
\beta_n(\pi\omega) \ge C_2 N_i, \ \text{whenever} \ P(\omega, \lambda, n) = \alpha_i, 1 \le i \le q(n)
\end{equation}
But for the measure of maximal entropy $\mu_{max}^+$ on $\Sigma_2^+$ we have $\mu_{max}^+([\omega_0, \ldots, \omega_{n-1}]) = \frac {1}{2^n}$, so from (\ref{betan}),
\begin{equation}\label{ni}
\int_{\Sigma_2^+} \log \beta_n(\pi \omega) \ d\mu_{max}^+(\omega) \ge \mathop{\sum}\limits_{j=1}^{q(n)} (\log C_2 N_j) \cdot \frac{N_j}{2^n} =  \log 2^n  + \mathop{\sum}\limits_{j=1}^{q(n)} \frac{N_j}{2^n}\log\frac{N_j}{2^n} + \log C_2
\end{equation}
However in general for any probability vector $(p_1, \ldots, p_{m})$, one has the upper bound (for eg \cite{Wa}), 
$$-\mathop{\sum}\limits_{i = 1 }^{m} p_i\log p_i \le \log m$$ 
From (\ref{vn}), we know $N_1 + \ldots N_{q(n)} = 2^n$, so we can take the probability vector $(\frac{N_1}{2^n}, \ldots, \frac{N_{q(n)}}{2^n})$, and from (\ref{ni}) it follows that:
$$\frac 1n \log \int_{\Sigma_2^+} \log \beta_n(\pi \omega) \ d\mu_{max}^+(\omega) \ge \log 2 - \frac{\log C_1\lambda^{-n}}{n} + \frac{\log C_2}{n}$$
This implies then from (\ref{otop}) that $o(\mathcal S_\lambda) \ge 2\lambda$, hence $o(\mathcal S_\lambda) > 1$ since $\lambda > \frac 12$.

\end{proof}

 \textbf{3.2.} We now look at examples with eventual exact or at least substantial overlaps, in which case the overlap number will be estimated, or even computed exactly. We look at the case when there are \textit{exact overlaps}, i.e. when we have $$\phi_{i_1\ldots i_p}(\Lambda) = \phi_{j_1\ldots j_p}(\Lambda),$$ for certain maximal tuples $(i_1, \ldots, i_p), (j_1, \ldots, j_p)$. Exact overlaps may appear after certain number of iterates, however for simplicity we look firstly at the case when $p=1$; the generalization is straightforward.  
 
So, consider the system $\mathcal S = \{\phi_i, 1 \le i \le m\}$ of conformal injective contractions, and assume we have the blocks 
\begin{equation}\label{blocks}
\phi_1 = \ldots = \phi_{k_1}, \ \phi_{k_1+1} = \ldots = \phi_{k_2}, \ldots, \ \phi_{k_p} = \phi_m,
\end{equation}
where there are no overlaps between the different blocks, i.e the system $\{\phi_{k_i}, 1 \le i \le p\}$ satisfies the Open Set Condition.

Let $\mu_{max}^+$  be the measure of maximal entropy on $\Sigma_m^+$, and denote the measure of maximal entropy for $\Phi$ on $\Sigma_m^+ \times \Lambda$ by $\mu_{max}$; then the overlap number $o(\mathcal S):= o(\mathcal S,\mu_{max})$ is in fact the topological overlap number which takes in consideration all preimages, and we proved in \cite{MU-JSP2016} that
\begin{equation}\label{8}
o(\mathcal S) = \exp\big(\mathop{\lim}\limits_{n \to \infty} \frac 1n \int_{\Sigma_m^+} \log \beta_n(\pi \omega) \ d\mu_{max}^+(\omega)\big),
\end{equation}
where $\beta_n(x):= \text{Card}\{(\eta_1, \ldots, \eta_n) \in I^n, \ x\in \phi_{\eta_1\ldots \eta_n}(\Lambda)\}$.
In this case, if $x \in \phi_{j_1\ldots j_n}(\Lambda)$ and if $k_{i_\ell -1}+1 \le j_\ell \le k_{i_\ell}$, then for $x = \pi\omega$ and $\omega= (j_1j_2 \ldots)$, we have:
\begin{equation}\label{8}
\beta_n(x) = (k_{i_1} - k_{i_1 -1}) \cdot \ldots \cdot (k_{i_n} - k_{i_n-1}),
\end{equation}
where if $i_\ell = 1$, then the factor $(k_{i_\ell}-k_{i_\ell -1})$ is replaced by $k_1$.
Let us take the function $\Psi: \Sigma_m^+ \to \mathbb R$, $\Psi(\omega):= \log k_1$ for $1 \le \omega_1 \le k_1$, and $\Psi(\omega):= \log(k_i - k_{i-1})$ for $k_{i-1}+1 \le \omega_1 \le k_i$. If $\omega, \eta$ are close enough in $\Sigma_m^+$, then $\omega_1 = \eta_1$, hence $\Psi$ is H\"older continuous on $\Sigma_m^+$. \newline
Notice that, if $\omega \in [j_1\ldots j_n]$ and $k_{i_s-1}+1 \le j_s \le k_{i_s}$ if $i_s > 1$, or $1 \le j_1 \le k_1$ if $i_s = 1$, then
$$\Psi(\omega) = \log(k_{i_1} - k_{i_1-1}), \ \Psi(\sigma \omega) = \log(k_{i_2} - k_{i_2-1}), \ldots$$
However from above,$$\int_{\Sigma_m^+} \log \beta_n(\pi \omega) \ d\mu_{max}^+(\omega) = \mathop{\sum}\limits_{s = 1, \ldots n} \mathop{\sum}\limits_{k_{i_s-1} +1 \le j_s \le k_{i_s}} \int_{[j_1 \ldots j_n]} \log(k_{i_1} - k_{i_1-1}) + \ldots \log (k_{i_n}-k_{i_n-1}) \  d\mu_{max}^+(\omega)$$
Thus, if $S_n\Psi$ denotes the consecutive sum of $\Psi$ with respect to $\sigma$, we obtain 
\begin{equation}\label{9}
\int_{\Sigma_m^+} \log \beta_n(\pi\omega) \ d\mu_{max}^+(\omega) = \int_{\Sigma_m^+} S_n\Psi(\omega) \ d\mu_{max}^+(\omega)
\end{equation}
Hence from (\ref{9}), by  Birkhoff Egodic Theorem for the measure of maximal entropy $\mu_{max}^+$ on $\Sigma_m^+$, $$\frac{1}{n} \int\log\beta_n(\pi\omega) \ d\mu_{max}^+(\omega) = \frac 1n\int_{\Sigma_m^+} S_n\Psi(\omega) \ d\mu_{max}^+(\omega) \mathop{\longrightarrow}\limits_{n \to \infty} \int_{\Sigma_m^+} \Psi(\omega) \ d\mu_{max}^+(\omega)$$
We have thus proved the following
\begin{pro}\label{o}
In the above setting from (\ref{blocks}), the topological overlap of the system $\mathcal S$ is given by 
$$o(\mathcal S) = o(\mathcal S, \mu_{max}) = \exp\big(\frac{k_1\log k_1 + (k_2 - k_1) \log(k_2 - k_1) + \ldots + (k_p-k_{p-1}) \log(k_p -k_{p-1})}{m}\big)$$
\end{pro}

As in Corollary \ref{o1}, the above estimates can be extended for the $p$-iterated system $\mathcal S^p = \{\phi_{i_1\ldots i_p}, \ i_j \in I, 1 \le j \le p\}$, and thus we obtain:
  
\begin{cor}\label{levelp}
Assume we have the system of conformal injective contractions $\mathcal S = \{\phi_i, i \in I\}$ with $|I| = m$, and let $\Lambda$ be its limit set. Assume also that there exists a family $\mathcal F \subset I^p$ of $p$-tuples such that $\phi_{i_p\ldots i_1}(\Lambda) = \phi_{j_p\ldots j_1}(\Lambda)$ for $(i_1, \ldots, i_p), (j_1, \ldots, j_p) \in \mathcal F$, and denote $Card(\mathcal F) = N(\mathcal F)$. Then $$o(\mathcal S) \ge \exp\big(\frac{N(\mathcal F)\log N(\mathcal F)}{m^p}\big)$$
\end{cor}

\

However, \'a priori there may exist only \textit{partial overlaps} at the level of $p$-iterates, which comprise a positive proportion of the measure. 
In particular the next Corollaries apply well for Bernoulli convolutions systems $\mathcal S_\lambda$, since in this case the limit set is an interval $\Lambda = I_\lambda$ and we can numerically estimate the proportion of overlaps at some iterate $p$. As above we obtain.
 
\begin{cor}\label{partial}
In the above setting assume that there is a family $\mathcal F\subset I^p$ of $p$-tuples and $k \ge 1$ so that for any $(i_1, \ldots, i_p) \in \mathcal F$, there exists $(j_1 \ldots j_k) \in I^k$ such that $$\phi_{i_1\ldots i_pj_1\ldots j_k}(\Lambda) \subset \mathop{\cap}\limits_{(\ell_1, \ldots \ell_p) \in \mathcal F} \phi_{\ell_1 \ldots \ell_p}(\Lambda)$$ Then if $N(\mathcal F)$ denotes the cardinality of $\mathcal F$, we obtain: $$o(\mathcal S) \ge \exp\big(\frac{N(\mathcal F) \log N(\mathcal F)}{m^{p+k}}\big)$$
\end{cor}

More generally we have the following:

\begin{cor}\label{partial2}
In the above setting assume that  there are families $\mathcal F_1, \ldots, \mathcal F_s \subset I^p$ of $p$-tuples and positive integers $k_1, \ldots, k_s$ such that, for any $1 \le j \le s$ and for any $(i_{j1}, \ldots, i_{jp}) \in \mathcal F_j$ there exists some $k_j$-tuple $(j_{1}, \ldots, j_{k_j}) \in I^{k_j}$ with $$\phi_{i_{j1}\ldots i_{jp}j_{1}\ldots j_{k_j}}(\Lambda) \subset \mathop{\cap}\limits_{(\ell_1, \ldots, \ell_p) \in \mathcal F_j} \phi_{\ell_1\ldots \ell_p}(\Lambda)$$ Then if $N(\mathcal F_j) := Card \mathcal F_j, \ 1\le j \le s$, we obtain:
$$o(\mathcal S) \ge \exp\big(\frac{N(\mathcal F_1)\log N(\mathcal F_1)}{m^{p+k_1}} + \ldots + \frac{N(\mathcal F_s)\log N(\mathcal F_s)}{m^{p+k_s}}\big)$$
\end{cor}

In Corollaries \ref{o1} and \ref{o2} below, we will apply these formulas  to box dimension estimates for the measure $\pi_{*}\mu_{max}^+$.
 
\

\textbf{3.3. Conditional measures associated to the lift in the general case.}

\

We now study several families of conditional measures associated to the lift $\Phi$ and to the equilibrium state $\mu:= \mu_\psi$ and various fiber partitions. We look at the relations between them, and find in particular a formula for the folding entropy.\newline Thus let  the following measurable partitions:

\ \textbf{i)} Consider the skew product map $\Phi: \Sigma_I^+ \times \Lambda \to \Sigma_I^+ \times \Lambda$ and its fibers $\Phi^{-1}(\omega, x)$ for $(\omega, x) \in \Sigma_I^+ \times \Lambda$. They form a  partition which is clearly measurable, and according to Rokhlin (\cite{Ro}) there exists a canonical family of conditional measures of $\mu:=\mu_\psi$, so for $\mu$-a.e $(\omega, x) \in \Sigma_I^+ \times \Lambda$, the conditional measure $\mu_{(\omega, x)}$ is supported on the finite set $\Phi^{-1}(\omega, x)$. Notice that $$\Phi^{-1}(\omega, x) = \{(i\omega, \phi_i^{-1}x), \ i \in I, \ \text{if} \ x \in \phi_i(\Lambda)\},$$ where we denote $\mu_{(\omega, x)}(i) := \mu_{(\omega, x)}(i\omega, \phi_i^{-1}x)$  if $x \in \phi_i(\Lambda)$, and $\mu_{(\omega, x)}(i\omega, \phi_i^{-1}(x)) = 0$ if $x \notin \phi_i(\Lambda)$. 

\ \textbf{ii)} Denote by $\mu^+:= \pi_{1*} \mu$ on $\Sigma_I^+$, where  $\pi_1: \Sigma_I^+ \times \Lambda \to \Sigma_I^+$ is the projection on the first coordinate, $\pi_1(\omega, x) = \omega$. Consider the partition of $\Sigma_I^+$ with the fibers of $\sigma$, and the associated family of conditional measures $\mu^+_\omega$ on the finite set $\sigma^{-1}\omega$, for $\mu^+$-a.e $\omega \in \Sigma_I^+$. We  also denote $\mu^+_\omega(i\omega)$ by $\mu^+_\omega(i)$.

\ \textbf{iii)} 
Consider the partition of $\Sigma_I^+ \times \Lambda$ with the fibers of $\pi_1: \Sigma_I^+ \times \Lambda \to \Sigma_I^+$, and the associated family of conditional measures of $\mu$, namely $\mu_\omega$ on $\pi_1^{-1}(\omega) = \{\omega\} \times \Lambda$, for $\mu$-a.e $\omega \in \Sigma_I^+$. So $\mu_\omega$ is actually a probability measure on $\Lambda$.

\

From \cite{MU-JSP2016} we know that, for an equilibrium measure $\mu_\psi$ on $\Sigma_I^+ \times \Lambda$, the overlap number is $$o(\mathcal S, \mu_\psi) = \exp(F_\Phi(\mu_\psi))$$
We prove now a formula, which gives  $F_\Phi(\mu)$ (and thus $o(\mathcal S, \mu_\psi))$ in terms of the conditional measures $\mu_\omega$ and $\mu_\omega^+$:

\begin{thm}\label{Hfold}
The overlap number $o(\mathcal S, \mu)$ of the equilibrium measure $\mu := \mu_\psi$ of a H\"older continuous potential on $\Sigma_I^+ \times \Lambda$, is determined by the corresponding conditional families $(\mu_\omega)_\omega, (\mu_\omega^+)_\omega$ by,
$$\log o(\mathcal S, \mu) = -\mathop{\sum}\limits_{i \in I} \int_{\Sigma_I^+ \times \Lambda}  \frac{\mu^+_\omega(i)}{\mathop{\sum}\limits_{j \in I} \mu^+_\omega(j) \cdot \mathop{\lim}\limits_{A_2 \to x} \frac{\mu_{j\omega}(\phi_j^{-1}\phi_i A_2)}{\mu_{i\omega}(A_2)}} \cdot \log \big(\frac{\mu^+_\omega(i)}{\mathop{\sum}\limits_{j \in I} \mu^+_\omega(j) \cdot \mathop{\lim}\limits_{A_2 \to x} \frac{\mu_{j\omega}(\phi_j^{-1}\phi_i A_2)}{\mu_{i\omega}(A_2)}}\Large) \ d\mu(\omega, x)
$$
\end{thm}

\begin{proof}
From the properties of conditional measures, if $\tilde g: \Sigma_I^+ \times \Lambda \to \mathbb R$ is $\mu$-integrable, then
\begin{equation}\label{1}
\aligned
\int_{\Sigma_I^+ \times \Lambda} &\tilde g(\omega, x) \ d\mu(\omega, x) = \int_{\Sigma_I^+ \times \Lambda} \int_{\Phi^{-1}(\omega, x)} \tilde g(\omega', x') d\mu_{(\omega, x)}(\omega', x') \ d\mu(\omega, x) \\
&= \mathop{\sum}\limits_{i \in I} \int_{\Sigma_I^+ \times \Lambda} \tilde g(i\omega, \phi_i^{-1}x) \cdot \mu_{(\omega, x)}(i) \ d\mu(\omega, x)
\endaligned
\end{equation}
Notice that since our IFS has overlaps, a point $x\in \Lambda$ may belong to several sets of type $\phi_i(\Lambda)$.
But $\mu$ also decomposes after the fibers of $\pi_1$, so for any real-valued function $\tilde g$ $\mu$-integrable on $\Sigma_I^+ \times \Lambda$, 
\begin{equation}\label{2}
\aligned
\int_{\Sigma_I^+ \times \Lambda} \tilde g(\omega, x) &\ d\mu(\omega, x) = \int_{\Sigma_I^+} \int_{\{\omega\} \times \Lambda} \tilde g(\omega, x) d\mu_\omega(x) \ d\mu^+(\omega) = 
 \int_{\Sigma_I^+} \Gamma(\omega) d\mu^+(\omega) \\ 
 &= \int_{\Sigma_I^+} \int_{\sigma^{-1}\omega}\Gamma(\omega') d\mu^+_\omega(\omega') \ d\mu^+(\omega) = \mathop{\sum}\limits_{i \in I} \int_{\Sigma_I^+} \Gamma(i\omega) \mu^+_\omega(i) \ d\mu^+(\omega)\\ 
 &= \mathop{\sum}\limits_{i \in I} \int_{\Sigma_I^+ \times \Lambda} \mu^+_\omega(i) \cdot \int_{\{i\omega\}\times \Lambda} \tilde g(i\omega, x) \ d\mu_{i\omega}(x) \ d\mu(\omega, x),
\endaligned
\end{equation}
where $\Gamma(\omega):= \int_{\{\omega\}\times \Lambda}\tilde g(\omega, x) d\mu_\omega(x)$. 
By taking $\tilde g$ such that $\tilde g|_{[j]} = 0$ for $j \ne i$, we obtain from (\ref{1}) and (\ref{2}) that:
\begin{equation}\label{3}
\int_{\Sigma_I^+ \times \Lambda} \tilde g(i\omega, \phi_i^{-1}x) \mu_{(\omega, x)}(i) \ d\mu(\omega, x) = \int_{\Sigma_I^+\times \Lambda} \mu^+_\omega(i) \cdot \int_{\{i\omega\}\times \Lambda} \tilde g(i\omega, x) d\mu_{i\omega}(x) \ d\mu(\omega, x)
\end{equation}

Let us take now $\tilde g = \chi_A$, where $A = A_1 \times A_2$ is the product of two Borelian sets, and $A_1 \subset [i] \subset \Sigma_I^+$. Then if $i\omega \in A_1$, we have $$\int_{\{i\omega\} \times \Lambda} \tilde g(i\omega, x) d\mu_{i\omega}(x) = \mu_{i\omega}(A_2)$$
Let us denote $A_1(i):= \{\omega \in \Sigma_I^+, \ i\omega \in A_1\}$. Thus, with the above choice of $\tilde g$, 
$$\int_{\Sigma_I^+ \times \Lambda} \tilde g(i\omega, \phi_i^{-1}x) \mu_{(\omega, x)}(i) \ d\mu(\omega, x) = \int_{A_1(i) \times \phi_i(A_2)} \mu_{(\omega, x)}(i) \ d\mu(\omega, x)$$
So from the last two displayed equalities and (\ref{3}), it follows that
\begin{equation}\label{4}
\int_{A_1(i) \times \phi_i(A_2)} \mu_{(\omega, x)}(i) \ d\mu(\omega, x) = \int_{A_1(i) \times \Lambda} \mu_{i\omega}(A_2) \cdot \mu_\omega^+(i) \ d\mu(\omega, x)
\end{equation}
Since $\mu = \mu_\psi$ is the equilibrium measure of a H\"older continuous potential, and since the Bowen balls in $\Sigma_I^+ \times \Lambda$ are of type $[\omega_1\ldots \omega_n] \times B(x, r_0)$, it follows that $\mu^+$ is a doubling measure on $\Sigma_I^+$. Hence from Borel Density Lemma (\cite{Pe}), if $A_1(i)$ is  a ball around some fixed $\bar \omega$ in $\Sigma_I^+$, we obtain:
\begin{equation}\label{5}
\aligned
\frac{1}{\mu^+(A_1(i))} \int_{A_1(i)}\int_{\phi_i A_2} \mu_{(\omega, x)}(i) d\mu_\omega(x) &\ d\mu(\omega) = \frac{1}{\mu^+(A_1(i))} \int_{A_1(i)\times \phi_iA_2}\mu_{(\omega, x)}(i) \ d\mu(\omega, x) \\ 
&\mathop{\longrightarrow}\limits_{A_1(i) \to \bar\omega}\  \int_{\phi_i (A_2)}\mu_{(\bar\omega, x)}(i) \ d\mu_{\bar \omega}(x)
\endaligned
\end{equation}
On the other hand, $\int_{A_1(i)\times \Lambda} \mu_{i\omega}(A_2) \cdot \mu^+_\omega(i) d\mu(\omega, x) = \int_{A_1(i)}\mu_{i\omega}(A_2) \mu^+_\omega(i) \ d\mu^+(\omega)$. Hence from Borel Density Lemma, for $\mu^+$-a.e $\bar \omega \in \Sigma_I^+$, 
$$
\frac{1}{\mu^+(A_1(i))}\int_{A_1(i)} \mu_{i\omega}(A_2) \cdot \mu_\omega^+(i) \ d\mu^+(\omega) \mathop{\longrightarrow}\limits_{A_1(i) \to \bar \omega}\mu_{i\bar\omega}(A_2) \cdot \mu_{\bar \omega}^+(i)$$
Therefore from (\ref{4}) and (\ref{5}) it follows that, for $\mu^+$-a.e $\bar\omega \in \Sigma_I^+$, 
\begin{equation}\label{6}
\int_{\phi_i(A_2)} \mu_{(\bar \omega, x)}(i) \ d\mu_{\bar \omega}(x) = \mu_{i\bar\omega}(A_2) \cdot \mu_{\bar \omega}^+(i)
\end{equation}

On the other hand from the $\Phi$-invariance of $\mu$, it follows  that $$\int_{\Sigma_I^+\times \Lambda}\tilde g(\omega, x) \ d\mu(\omega, x) = \int_{\Sigma_I^+\times \Lambda} \tilde g\circ \Phi(\omega, x) d\mu(\omega, x) = \int_{\Sigma_I^+\times \Lambda} \tilde g(\sigma\omega, \phi_{\omega_1} x) d\mu(\omega, x)$$
Hence using the conditional decomposition of $\mu$ along the fibers  of $\pi_1$,
$$
\int_{\Sigma_I^+}\int_{\{\omega\}\times \Lambda} \tilde g(\omega, x) d\mu_\omega(x) \ d\mu^+(\omega) = \int_{\Sigma_I^+} \int_{\{\omega\}\times \Lambda} \tilde g(\sigma\omega, \phi_{\omega_1}x) d\mu_\omega(x) \ d\mu^+(\omega)$$ 
Let us take now again $\tilde g = \chi_{A_1 \times A_2}$, and notice that $\sigma \omega \in A_1$ and $\phi_{\omega_1}x \in A_2$, if and only if $\omega \in \sigma^{-1}A_1$ and $ x \in \phi_{\omega_1}^{-1}A_2$. So from above, 
$$\int_{A_1}\mu_\omega(A_2) \ d\mu^+(\omega) = \int_{\sigma^{-1}A_1} \mu_\omega(\phi_{\omega_1}^{-1}A_2) \ d\mu^+(\omega)$$
Since $\mu^+$ is $\sigma$-invariant on $\Sigma_I^+$, it follows then that:
$$\int_{\sigma^{-1}A_1}\mu_{\sigma\omega}(A_2) \ d\mu^+(\omega) = \int_{\sigma^{-1}A_1} \mu_\omega(\phi_{\omega_1}^{-1}A_2) \ d\mu^+(\omega)$$
Taking $A_1 \to \omega$, we obtain from above that, for any Borelian set $A_2\subset \Lambda, \ i \in I$ and $\mu^+$-a.e $\omega \in \Sigma_I^+$, 
\begin{equation}\label{6'}
\mu_\omega(\phi_i A_2) = \mathop{\sum}\limits_{j \in I} \mu_{j\omega}(\phi_j^{-1}\phi_i(A_2)) \cdot \mu^+_{\omega}(j)
\end{equation}
But we can apply Borel Density Lemma for the measure $\phi_*\mu_\omega$ on $\phi_i(\Lambda)$ in (\ref{6}), and we see that for any $x\in \phi_i(\Lambda)$ and any $r>0$ small, $B(x, r) \cap \phi_i\Lambda = \phi_i(B(\phi_i^{-1}x, r') \cap \Lambda)$ for some $r'>0$ since $\phi_i$ is injective. Thus by taking $A_2$ to be a neighbourhood of $x$, we obtain from (\ref{5}), (\ref{6}), (\ref{6'}), that $ \mathop{\lim}\limits_{A_2 \to x} \frac{\mu_{j\omega}(\phi_j^{-1}\phi_i A_2)}{\mu_{i\omega}(A_2)}$ exist, and that for $\mu$-a.e $(\omega, x) \in \Sigma_I^+ \times \Lambda$ and any $i\in I$, 
\begin{equation}\label{7}
\mu_{(\omega, x)}(i) = \frac{\mu^+_\omega(i)}{\mathop{\sum}\limits_{j \in I} \mu^+_\omega(j) \cdot \mathop{\lim}\limits_{A_2 \to x} \frac{\mu_{j\omega}(\phi_j^{-1}\phi_i A_2)}{\mu_{i\omega}(A_2)}}
\end{equation}
So from (\ref{7})  and the fact that  $F_\Phi(\mu) = -\int_{\Sigma_I^+\times \Lambda} \mu_{(\omega, x)} \log \mu_{(\omega, x)}d\mu(\omega, x)$, we obtain the formula for the folding entropy $F_\Phi(\mu)$, and thus from (\ref{otop}) the formula  for the overlap number $o(\mathcal S, \mu)$.

\end{proof}

\section{Box dimension estimates.}

For $\mu$ a Borel finite measure on $\mathbb R^d$, recall  (\cite{Pe}) that the \textit{lower box dimension} of $\mu$ is:
$$
\underline{dim}_B(\mu) = \mathop{\lim}\limits_{\delta \to 0} \inf\{\underline{dim}_B(Z), \mu(Z) \ge 1-\delta\}$$
Also denote the Hausdorff dimension of $\mu$ by $HD(\mu)$. The following inequality holds (see \cite{Pe}),
$$HD(\mu) \le \underline{dim}_B(\mu) $$

In the sequel  denote by $\chi(\mu_\psi)$ the Lyapunov exponent of the measure $\mu_\psi$ on $\Sigma_I^+ \times \Lambda$. Some aspects of dimensions and measures for various other cases were studied in \cite{M-MZ}, \cite{M-ETDS11}, etc.
We are now ready to prove the estimate for the lower box dimension of the projection $\nu_\psi:= \pi_{2*}(\mu_\psi)$; recall that $\nu_\psi$ is \textit{not} the usual projection measure $\pi_*\pi_{1*}\mu_\psi$. The following Theorem gives a \textit{constructive method} to obtain sets $Z$ of large $\nu_\psi$-measure whose box dimensions is estimated using overlap numbers, and an estimate of the number of balls needed to cover such sets $Z$.

\begin{thm}\label{mthm}
Consider the conformal IFS $\mathcal S = \{\phi_i, i \in I\}$ wih limit set $\Lambda$,  and the H\"older continuous potential $\psi: \Sigma_I^+ \times \Lambda \to \mathbb R$, with its equilibrium measure $\mu_\psi$, and let $\nu_\psi:= \pi_{2*}\mu_\psi$. Then, 
$$\underline{dim}_B(\nu_\psi) \le \frac{h_\sigma(\pi_{1*}(\mu_\psi)) - \log o(\mathcal S, \mu_\psi))}{|\chi(\mu_\psi)|}$$

\end{thm}

\begin{proof}
For $n \ge 1$, let $S_n\psi(\omega, x):= \psi(\omega, x) + \psi(\Phi(\omega, x)) + \ldots + \psi(\Phi^{n-1}(\omega, x))$. For all $(\omega, x) \in \Sigma_I^+ \times \Lambda$, $\Phi^n(\omega, x) = (\sigma^n(\omega), \phi_{\omega_n\ldots \omega_1}(x))$. 
From Chain Rule, $J_{\Phi^n}(\mu_\psi)(\omega, x) = J_\Phi(\mu_\psi)(\omega, x) \ldots J_\Phi(\mu_\psi)(\Phi^{n-1}(\omega, x))$. 
We know from the Birkhoff Ergodic Theorem, from the formula for folding entropy  (\ref{foldingentropy}) and the fact that $\mu_\psi$ is ergodic that, $$\frac{1}{n}\log|\phi'_{\omega_n\ldots\omega_1}(x)| \mathop{\to}\limits_{n \to \infty}  \int_{\Sigma_I^+ \times \Lambda} \log|\phi'_{\omega_1}(x)| d\mu_\psi(\omega, x),  \ \text{and} \ \frac 1n \log J_{\Phi^n}(\mu_\psi)(\omega, x) \mathop{\to}\limits_{n \to \infty} F_\Phi(\mu_\psi), \  \text{and}$$   $$\frac{1}{n}S_n\psi(\omega, x) \mathop{\to}\limits_{n \to \infty} \int_{\Sigma_I^+ \times \Lambda} \psi(\omega, x) d\mu_\psi(\omega, x)$$  
For an integer $n \ge 1$ and an arbitrary number $\tau >0$, consider therefore the Borelian set 
$$
\aligned
D_n(\tau) := \{&(\omega, x) \in \Sigma_I^+ \times \Lambda,  \text{with} \  |\frac{1}{p}\log J_{\Phi^p}(\mu_\psi)(\omega, x)- F_\Phi(\mu_\psi)| < \tau, \ \text{and} \\  
&|\frac 1p \log|\phi'_{\omega_p\ldots\omega_1}(x)| - \int\log|\phi'_{\omega_1}|(x) d\mu_\psi(\omega, x)| < \tau, 
  \  \ |\frac 1p S_p\psi(\omega, x) - \int\psi d\mu_\psi | < \tau, \ \forall p \ge n\}
\endaligned
$$
From above, $\mu_\psi(D_n(\tau)) \mathop{\to}\limits_{n\to \infty} 1$ for all $\tau >0$, and moreover, 
\begin{equation}\label{Dn}
D_1(\tau) \subset \ldots \subset D_n(\tau) \subset D_{n+1}(\tau) \subset \ldots
\end{equation}
On the other hand, notice that a Bowen ball in $\Sigma_I^+ \times \Lambda$ has the form $[\omega_1\ldots \omega_n] \times B(x, r_0)$, and from the estimates of equilibrium measures on Bowen balls (for eg \cite{KH}), we have:
$$\mu_\psi([\omega_1 \ldots \omega_n] \times B(x, r_0)) \approx \exp(S_n\psi(\omega, x) - n P_\Phi(\psi)), n \ge 1, $$
where $\approx$ means that the two quantities are comparable with a comparability constant which depends only on $\psi$ and is independent of $n, x, \omega$. 
Now, if $\omega' \in [\omega_1 \ldots \omega_n]$ and if $(\eta, y) \in \Phi^{-n}\Phi^n(\omega, x)$, then $(\eta, y) \in \Phi^{-n}\Phi^n(\omega', x)$, and viceversa. But we proved in \cite{MU-JSP2016} that for $\mu_\psi$-a.e $(\omega, x) \in \Sigma_I^+ \times \Lambda$,
\begin{equation}\label{J}
J_{\Phi^n}(\mu_\psi)(\omega, x) \approx \frac{\mathop{\sum}\limits_{(\eta, y) \in \Phi^{-n}\Phi^n(\omega, x)} e^{S_n\psi(\eta, y)}}{e^{S_n\psi(\omega, x)}},
\end{equation}
with comparability constant independent of $\omega, x, n$. Therefore, if $\omega' \in [\omega_1 \ldots\omega_n]$, it follows from (\ref{J}) that there exists a constant $C>0$ so that for $\mu_\psi$-a.e $(\omega, x)$ and all $n \ge 1$,

\begin{equation}\label{omega'}
\frac 1C J_{\Phi^n}(\mu_\psi)(\omega', x) \le J_{\Phi^n}(\mu_\psi)(\omega, x) \le C J_{\Phi^n}(\mu_\psi)(\omega', x)
\end{equation}
This means that  $D_n(\tau)$ is basically a product set, or more precisely that there exists a set $E_n(\tau)\subset \Lambda$ such that $D_n(\tau/2) \subset [\omega_1\ldots\omega_n]\times E_n(\tau) \subset D_n(\tau)$.
Notice now that the map $\Phi^n$ is injective on the set $[\omega_1 \ldots \omega_n] \times B(x, r_0)$, for some fixed $r_0$, since the composition map $\phi_{\omega_n\ldots\omega_1}$ is injective on $U$.
Thus from the properties of Jacobians of measures on sets of injectivity, we get
\begin{equation}\label{iteratemu}
\aligned
\mu_\psi(\Phi^n([\omega_1\ldots\omega_n]&\times B(x, r_0)\cap D_n(\tau))) = \int_{[\omega_1\ldots \omega_n]\times B(x, r_0)\cap D_n(\tau)} J_{\Phi^n}(\mu_\psi)(\eta, y) \ d\mu_\psi(\eta, y) \\
&\ge C e^{n(F_\Phi(\mu_\psi)-\tau)}\cdot \mu_\psi\big([\omega_1\ldots \omega_n]\times B(x, r_0) \cap D_n(\tau)\big)
\endaligned
\end{equation}

We now want to estimate $\mu_\psi\big([\omega_1\ldots \omega_n]\times B(x, r_0) \cap D_n(\tau)\big)$.
First notice that, since $\psi$ is H\"older continuous, the consecutive sum $S_n\psi(\omega, x)$ with respect to $\Phi$, does not really depend on $x$, but only on $\omega$. So there exists a constant $C>0$ such that for any $x, y \in \Lambda, \ \omega \in \Sigma_I^+$,
$$|S_n\psi(\omega, x) - S_n\psi(\omega, y)| \le C$$
Thus one can fix $ y = x_0$ above in $\Lambda$. We want to show that for any Borel set $A \subset \Lambda$ and any $n$,
\begin{equation}\label{prod}
\mu_\psi([\omega_1\ldots\omega_n]\times A) \approx e^{S_n\psi(\omega, x_0)-nP_\Phi(\psi)}\cdot \nu_\psi(A),
\end{equation}
with comparability constants independent of $\omega, n , A$.
Since $\mu_\psi$ is a Borel measure, it is enough to show (\ref{prod}) for open balls $A = B(y, r)$. Let us also recall that all the contractions $\phi_i$ are conformal, thus we have a Bounded Distortion property on Bowen balls of $\Phi$, namely there exists constants $C>0,  0 < r_0 < 1$, such that for any $x, y \in \Lambda$ with $d(x, y) < r_0$, any integer $n$ and any sequence $\underline i \in \Sigma_I^+$, then, 
\begin{equation}\label{BDPR}
C^{-1} \phi_{i_1\ldots i_n}'(x) \le \phi_{i_1\ldots i_n}'(y) \le C \phi_{i_1\ldots i_n}'(x)
\end{equation}
 From the $\Psi$-invariance of $\mu_\psi$, we know that $\mu_\psi([\omega_1\ldots\omega_n]\times B(y, r)) = \sum \mu_\psi([\eta_1\ldots\eta_p\omega_1\ldots\omega_n] \times \phi_{\eta_1}^{-1}\ldots \phi_{\eta_p}^{-1}(B(y, r)))$, and using the above Bounded Distortion property, we take these backward iterates of $B(y, r)$ until we reach diameter $r_0$. The Bowen balls for the map $\Phi$ are sets  of type $[\omega_1 \ldots \omega_n] \times B(z, r_0)$. Then from the properties of equilibrium measures on Bowen balls (see \cite{KH}), $$\mu_\psi([\eta_1\ldots \eta_p\omega_1\ldots \omega_n]\times B(z, r_0)) \approx e^{S_{n+p}\psi(\eta_1\ldots \eta_p\omega_1\ldots \omega_n, z) - (n+p) P_\Phi(\psi)},$$
where the comparability constants do not depend on $z, \omega, \eta, n, p$. We will write also $S_n\psi(\omega_1\ldots\omega_n)$ for $S_n\psi(\omega, x_0)$, since from the above it does not matter (up to a constant) which $x_0$ we take.
But $$
\begin{aligned}
e^{S_{n+p}\psi(\eta_1\ldots \eta_p\omega_1\ldots \omega_n, z) - (n+p) P_\Phi(\psi)} &= e^{S_p\psi(\eta_1\ldots\eta_p) + S_n\psi(\omega_1\ldots \omega_n) - (n+p) P_\Phi(\psi)} = \\ 
&= e^{S_p\psi(\eta_1\ldots \eta_p) - p P_\Phi(\psi)} \cdot e^{S_n\psi(\omega_1\ldots \omega_n) - nP_\Phi(\psi)}
\end{aligned}
$$
So when we take the above sum we  obtain $e^{S_n\psi(\omega_1\ldots \omega_n) - nP_\Phi(\psi)} \cdot \nu_\psi(B(y, r))$.
Thus relation  (\ref{prod}) holds, i.e. there exists a  constant $C>0$ independent of $n, x, y, r, \omega$, such that $$\frac 1C \nu_\psi(B(y, r)) e^{S_n\psi(\omega, x_0) - nP_\Phi(\psi)} \le \mu_\psi([\omega_1\ldots \omega_n] \times B(y, r)) \le C \nu_\psi(B(y, r)) e^{S_n\psi(\omega, x_0) - nP_\Phi(\psi)}$$

Now recall that $\mu_\psi(D_n(\tau)) \to 1$ when $n \to \infty$; hence for any $\delta>0$ small, there exists $n(\delta)\ge 1$ such that $\mu_\psi(D_n(\tau)) \ge 1-\delta$ for all $n \ge n(\delta)$; hence from the $\Phi$-invariance of $\mu_\psi$, 
$\mu_\psi(\Phi^{n}(D_n(\tau))) \ge 1-\delta$. Moreover there exists a strictly increasing sequence of integers $(k_n)_n$, with $k_n \ge n$,  such that,
 \begin{equation}\label{kn}
\mu_\psi(D_{k_n}) \ge 1-\alpha_n, \ \text{and} \ \mathop{\sum}\limits_{n\ge 1} \alpha_n < \infty
\end{equation}

Denote now by $Y_n(\tau):= \pi_2D_n(\tau) \subset \Lambda$.
We want to apply a version of Borel Density Lemma (\cite{Pe} pg 293), in order to estimate the portion of the $\nu_\psi$-measure of the intersection between a ball and $ Y_n(\tau)$. Indeed for any $\delta>0$ it follows that for any $n \ge n(\delta)$, there exists a borelian subset $\tilde Y_n(\tau) \subset Y_n(\tau)$ and $\rho_n>0$, such that $\nu_\psi(\tilde Y_n(\tau)) \ge 1-2\delta$, and for any $x \in \tilde Y_n(\tau)$ and any $r \le \rho_n$, 
\begin{equation}\label{tilde}
\nu_\psi(B(x, r) \cap  Y_n(\tau)) \ge \frac 12 \nu_\psi(B(x, r))
\end{equation}
Let  $Z_n(\tau):= \pi_2\Phi^n(D_n(\tau))$ and $\tilde Z_n(\tau):= \mathop{\bigcap}\limits_{\ell \ge n} Z_{k_\ell}(\tau)$, for $n \ge 1$. Then, since $\mu_\psi(\Phi^n(D_n(\tau))) \ge \mu_\psi(D_n(\tau))$, it follows from (\ref{kn}) that 
$$\nu_\psi(\tilde Z_{n}(\tau)) \ge 1-\mathop{\sum}\limits_{m \ge n}\alpha_m, \ \text{and} \ \nu_\psi(\tilde Z_n(\tau)) \mathop{\to}\limits_{n\to \infty} 1$$
Given the radius $\rho_n$ above, we can find an integer $s_n \ge n$, such that any ball $B(y, \frac{\rho_n}{2})$ with $y \in \Lambda$, intersects the set $\tilde Y_{s_n}(\tau)$. This is true since $\nu_\psi(\tilde Y_n(\tau)) \to 1$, and since $\mu_\psi$ is the equilibrium measure of a H\"older continuous potential, thus it is positive on balls of radius $\rho_n/2$. 
Denote now 
$$r_n:= e^{n(\chi(\mu_\psi) + \tau)}, \ n \ge 1
$$ 

Consider an arbitrary family $\mathcal F_{k_\ell}$ of mutually disjoint balls of radii $\rho_n r_{k_\ell}$ with centers in $\pi_2\Phi^{k_\ell}(D_{k_\ell}(\tau))$, for $\ell \ge s_n$, and assume the balls in $\mathcal F_{k_\ell}$ contain images of type $\phi_{i_{k_\ell}\ldots i_1}(B(z, \rho_n))$ for $z$ in a family of centers $F_{k_\ell}$. But from above, for all $\ell\ge s_n$ and $z \in F_{k_\ell}$,  the  ball $B(z, \rho_n/2)$  must contain a point $\xi_z \in \tilde Y_{s_n}(\tau)$.  Hence $B(\xi_z, \rho_n/2) \subset B(z, \rho_n)$, and thus $\phi_{i_{k_\ell}\ldots i_1}(B(\xi_z, \rho_n/2)) \subset \phi_{i_{k_\ell}\ldots i_1}(B(z, \rho_n))$ for all $z \in F_{k_\ell}$. So we obtain a family $\mathcal G_{k_\ell}$ of disjoint sets $\phi_{i_{k_\ell}\ldots i_1}(B(\xi_z, \rho_n/2)), \ z\in F_{k_\ell}$. From our construction,  $$N(\mathcal G_{k_\ell}):= Card(\mathcal G_{k_\ell}) = N(\mathcal F_{k_\ell}):= Card(\mathcal F_{k_\ell})$$
However $\tilde Y_{s_n}(\tau) \subset Y_{s_n}(\tau) \subset \pi_2 D_{k_\ell}(\tau)$, if $\ell \ge s_n \ge n$, so from the above properties of the set $\tilde Y_{s_n}(\tau)$ and (\ref{tilde}), it follows that $\nu_\psi(\tilde Y_{s_n}(\tau)) \ge 1-2\delta$ and,
$$
\nu_\psi(B(\xi_z, \rho_n/2) \cap Y_{s_n}(\tau)) \ge \frac 12 \nu_\psi(B(\xi_z, \rho_n/2)) 
$$
But  now from (\ref{Dn}), $Y_{s_n}(\tau) \subset Y_k(\tau) = \pi_2D_{k}(\tau)$ for all $k \ge s_n$, and recall $\ell \ge s_n \ge n$; hence from the last inequality, 
\begin{equation}\label{meas}
\nu_\psi(B(\xi_z, \rho_n/2) \cap Y_{k_\ell}(\tau)) \ge \frac 12 \nu_\psi(B(\xi_z, \rho_n/2))
\end{equation}

Let us estimate now the $\nu_\psi$-measure of a set from $\mathcal G_{k_\ell}$, for $\ell \ge s_n$. Since $\Phi^{k_\ell}$ is injective on $[i_1\ldots i_{k_\ell}]\ \times \Lambda$, we obtain from (\ref{prod}) and (\ref{meas}),
\begin{equation}\label{mare}
\aligned
\nu_\psi(\phi_{i_{k_\ell}\ldots i_1}&B(\xi_z, \rho_n/2)\cap Y_{k_\ell}(\tau)) = \mu_\psi(\Phi^{k_\ell}([i_1\ldots i_{k_\ell}] \times B(\xi_z, \rho/2) \cap D_{k_\ell})) = \\
& =\int_{[i_1\ldots i_{k_\ell}]\times (B(\xi_z, \rho_n/2) \cap Y_{k_\ell}(\tau))} J_{\Phi^{k_\ell}}(\mu_\psi)(\omega, x) \ d\mu_\psi(\omega, x)\\
& \ge C \exp(k_\ell(F_\Phi(\mu_\psi) - \tau)) \cdot \exp(S_{k_\ell}\psi(\omega, x_0) - k_\ell P_\Phi(\psi))\cdot \nu_\psi(B(\xi_z, \rho_n/2) \cap Y_{k_\ell})\\
& \ge \tilde C_n\exp(k_\ell(F_\Phi(\mu_\psi)-\tau) \cdot \exp(k_\ell(-h_\Phi(\mu_\psi) - \tau)) = \tilde C_n\exp(k_\ell(F_\Phi(\mu_\psi)-h_\Phi(\mu_\psi) -2\tau)),
\endaligned
\end{equation}
for some constants $C_n, \tilde C_n>0$, where we used the estimate on the Jacobian of $\Phi^{k_\ell}$ on $D_{k_\ell}$,  the estimate on the equilibrium measure $\mu_\psi$ of a Bowen ball $[i_1\ldots i_{k_\ell}]\times B(\xi_z, \rho_n/2)$, and  the behaviour of $S_{k_\ell}\psi$ on the generic points from $D_{k_\ell}$. \ 
Since the balls in $\mathcal F_{k_\ell}$ are disjoint, and each of them contains a set  of type $\phi_{i_{k_\ell}\ldots i_1}B(\xi_z, \rho_n/2)\cap Y_{k_\ell}$, it follows that for all  integers $\ell \ge s_n$, 
$$\mathop{\sum}\limits_{\xi_z\in G_{k_\ell}} \nu_\psi(\phi_{i_{k_\ell}\ldots i_1}B(\xi_z, \rho_n/2)\cap Y_{k_\ell}(\tau)) \le 1$$
Thus, using (\ref{mare}) and the fact that $N(\mathcal G_{k_\ell}) = N(\mathcal F_{k_\ell})$, we obtain for any family $\mathcal F_{k_\ell}$ as above,
\begin{equation}\label{NF}
N(\mathcal F_{k_\ell}) \le C_n^{-1}\exp(-k_\ell(F_\Phi(\mu_\psi)- h_\Phi(\mu_\psi) -2\tau))
\end{equation}
So if for some $\ell \ge s_n$ we take a disjointed family $\mathcal W$ of balls of radii $\rho_n\cdot r_{k_\ell}$ with centers in $\tilde Z_{s_n}(\tau) = \mathop{\bigcap}\limits_{\ell \ge s_n} \pi_2\Phi^{k_\ell}D_{k_\ell}(\tau)$, then its cardinality $N(\mathcal W)$ is less than the cardinality of some family $\mathcal F_{k_\ell}$ from above, hence from (\ref{NF}) we obtain an estimate for the lower box dimension, 
$$\underline{dim}_B(\tilde Z_{s_n}(\tau)) \le \frac{h_\Phi(\mu_\psi) - F_\Phi(\mu_\psi) + 2\tau}{|\chi(\mu_\psi) + \tau|}$$

But on the other hand, we know from construction that $\nu_\psi(\tilde Z_{s_n}(\tau)) \ge 1- \mathop{\sum}\limits_{j\ge s_n} \alpha_j \to 1$, when $n \to \infty$. So from the above, using the definition of lower box dimension of a measure, it follows $$\underline{dim}_B(\nu_\psi) \le \frac{h_\Phi(\mu_\psi) - F_\Phi(\mu_\psi) + 2\tau}{|\chi(\mu_\psi) + \tau|},$$ for any small number $\tau >0$, and thus the conclusion follows, namely $$\underline{dim}_B(\nu_\psi) \le \frac{h_\Phi(\mu_\psi) - F_\Phi(\mu_\psi)}{|\chi(\mu_\psi)|}$$
\end{proof}

Recall now from (\ref{bernoulli}) that for Bernoulli measures we have the equality of the two projectional measures, i.e \ $\pi_{2*} \mu_{\bf p} = \pi_* \mu_{\bf p}^+$. Also recall that $\mu_{max}$ is the measure of maximal entropy for $\Phi$ on $\Sigma_I^+ \times \Lambda$, and $\mu_{max}^+$ is the measure of maximal entropy for the shift on $\Sigma_I^+$. 

Then from Proposition \ref{o}, Theorem \ref{mthm} and  Corollaries \ref{levelp} and \ref{partial}, we obtain the following estimates. In particular, these can be applied to Bernoulli convolutions (for which the limit set $\Lambda$ is the whole interval $I_\lambda$), to get \textit{numerical estimates} of the box dimension of the projection  measure, based on how many overlaps we count at level $p$ and on how large are these overlaps.  

\begin{cor}\label{o1}
Assume we have the system of conformal injective contractions $\mathcal S = \{\phi_i, i \in I\}$ with $|I| = m$, and let $\Lambda$ be its limit set, and denote by $\mu_{max}$ the measure of maximal entropy on $\Sigma_I^+ \times \Lambda$. Assume also that there exists a family $\mathcal F$ of $p$-tuples such that $\phi_{i_p\ldots i_1}(\Lambda) = \phi_{j_p\ldots j_1}(\Lambda)$ for $(i_1, \ldots, i_p), (j_1, \ldots, j_p) \in \mathcal F$, and denote $Card(\mathcal F) = N(\mathcal F)$. Then $o(\mathcal S) \ge \exp\big(\frac{N(\mathcal F)\log N(\mathcal F)}{m^p}\big)$, and $$\underline{dim}_B(\pi_{2*}\mu_{max}) = \underline{dim}_B(\pi_*\mu^+_{max}) \le 
\frac{p \cdot h_\sigma(\mu_{max}^+) - \frac{N(\mathcal F)\log N(\mathcal F)}{m^p}}{p\cdot \chi(\mu_{max})}$$
\end{cor}

\begin{cor}\label{o2}

In the above setting assume that  there are families $\mathcal F_1, \ldots, \mathcal F_s \subset I^p$ of $p$-tuples and positive integers $k_1, \ldots, k_s$ such that, for any $1 \le j \le s$ and any $(i_{j1}, \ldots, i_{jp}) \in \mathcal F_j$ there exists some $k_j$-tuple $(j_{1}, \ldots, j_{k_j}) \in I^{k_j}$, with $$\phi_{i_{j1}\ldots i_{jp}j_{1}\ldots j_{k_j}}(\Lambda) \subset \mathop{\cap}\limits_{(\ell_1, \ldots, \ell_p) \in \mathcal F_j} \phi_{\ell_1\ldots \ell_p}(\Lambda)$$ Then if $N(\mathcal F_j) := Card \mathcal F_j, \ 1\le j \le s$, we obtain:
$$
\underline{dim}_B(\pi_{2*}\mu_{max}) = \underline{dim}_B(\pi_*\mu^+_{max}) \le 
\frac{p \cdot h_\sigma(\mu_{max}^+) - \frac{N(\mathcal F_1)\log N(\mathcal F_1)}{m^{p+k_1}} - \ldots - \frac{N(\mathcal F_s)\log N(\mathcal F_s)}{m^{p+k_s}}}{{p\cdot \chi(\mu_{max})}}
$$
\end{cor}

\

\

\textbf{Acknowledgements:} This work was supported by grant PN-III-P4-ID-PCE-2016-0823 from  CNCS - UEFISCDI.

\

Address: \ Institute of Mathematics ``Simion Stoilow`` of the Romanian Academy, 

Calea Grivitei 21,  P.O. Box 1-764, RO 014700, Bucharest, Romania.

 Eugen.Mihailescu\@@imar.ro
 
www.imar.ro/$\sim$mihailes


\begin{thebibliography}{99}

\bibitem{Ba} L. Barreira, Dimension and Recurrence in Hyperbolic Dynamics, Birkh\"auser, Basel, 2008.





\bibitem{Ga}
A. Garsia, Arithmetic properties of Bernoulli convolutions, Trans. AMS, 102, 1962, 409-432.

\bibitem{KH}
A. Katok, B. Hasselblatt, Introduction to the Modern Theory
of Dynamical Systems, Cambridge Univ. Press, London-New York,
1995.

\bibitem{KS} I.P. Kornfeld, Ya.G. Sinai, Basic Notions of Ergodic Theory and Examples of Dynamical Systems, in: Ya.G. Sinai (Ed.), Dynamical Systems, Ergodic Theory and Applications, in: Encyclopaedia Math. Sci., vol.100, Springer Verlag, 2000.

\bibitem{M-MZ}
E. \ Mihailescu, Unstable directions and fractal dimension for a class of skew products with overlaps in fibers, Math Zeitschrift, 269, 2011, 733-750. 

\bibitem{M-ETDS11}
E. Mihailescu,  On a class of stable conditional measures, Ergod Th Dyn Syst 31, 2011, 1499-15.


\bibitem{MU-JSP2016}
E. Mihailescu, M. Urba\'nski, Overlap functions for measures in conformal iterated function systems, 
J. Statistical Physics, 162, 2016, 43-62.

\bibitem{Pa}
W.  Parry, Entropy and generators in ergodic theory, W. A
Benjamin, New York, 1969.


\bibitem{Pe}
Y. Pesin, Dimension Theory in Dynamical Systems, Chicago Lectures in Mathematics, 1997.

\bibitem{PU}
F. Przytycki, M. Urba\'nski,  On Hausdorff dimension of some fractal sets. Studia Math. 93, 155-186, 1989.

\bibitem{Ro}
V. A. Rokhlin, Lectures on the theory of entropy of
transformations with invariant measures, Russian Math. Surveys,
\textbf{22}, 1967, 1-54.

\bibitem{Ru-survey99}
D. Ruelle, Smooth dynamics and new theoretical ideas in
nonequilibrium statistical mechanics, J. Statistical Physics
\textbf{95}, 1999, 393-468.

\bibitem{Ru-folding}
D.  Ruelle, Positivity of entropy production in nonequilibrium
statistical mechanics, J. Statistical Physics \textbf{85}, 1/2,
1996, 1-23.

\bibitem{Wa}
P.  Walters, An introduction to ergodic theory (2nd edition),
Springer New York, 2000.
\end{thebibliography}
\end{document}